\documentclass[reqno,english]{amsart}

\usepackage{packages}
\usepackage{macros}

\usepackage{natbib}

\newcommand*{\dreg}{\mathcal{G}_{n,\Delta}}

\newcommand*{\net}{\mathcal{N}}
\newcommand*{\netB}{\mathcal{N_{\mathcal B}}}
\newcommand*{\seeds}{\mathcal{S}}
\newcommand*{\vx}{\mathbf{x}}
\newcommand*{\hx}{\hat{\vx}}

\newcommand*{\tuples}{\Xi}
\newcommand*{\ctail}{c_{\text{\tiny{\ref{thm:graphs}}}}}

\def\Vol{{\rm Vol}}

\def\calL{{\mathcal{L}}}

\newtheorem*{theorem*}{Theorem}
\newtheorem*{problem}{Problem}

\title[Metric dimension reduction modulus]{Metric dimension reduction modulus \\ for superlogarithmic distortion}

\author[D.\ J.\ Altschuler \and K.\ Tikhomirov]{
        Dylan J. Altschuler \and 
        Konstantin Tikhomirov 
        }

\address{Dylan J. Altschuler, Department of Mathematical Sciences, Carnegie Mellon University.}
\address{Konstantin Tikhomirov, Department of Mathematical Sciences, Carnegie Mellon University.}
\begin{document}

\begin{abstract}
The metric dimension reduction modulus $k^\alpha_n(\ell_\infty)$ is the smallest $k$ such that every $n$--point metric space can be embedded into some $k$-dimensional normed space, with bi--Lipschitz distortion at most $\alpha$. Determining sharp asymptotics for $k^\alpha_n(\ell_\infty)$ is a fundamental task in metric geometry,
with $\alpha=\Theta(\log n)$ bearing particular interest. A line of advances over the past decades has led to an upper bound on 
$k^{\alpha}_n(\ell_\infty)$ for $\alpha = \Omega(\log n)$, but a matching lower bound has remained open. We close this gap, establishing: for every fixed $\beta > 0$, 
$$
k^{\alpha}_n(\ell_\infty) =\Theta\bigg(\frac{\log n}{\log(\frac{\alpha}{\log n}+1)}\bigg)\quad
\mbox{for every $\alpha\geq \beta \log n$}.
$$
This resolves a question from Naor's 2018 ICM plenary lecture. 
Our result is obtained by characterizing the minimum dimension $d$ for which, with high probability, a random regular graph admits an $\alpha$--embedding into some $d$--dimensional normed space.
\end{abstract}

\maketitle
\section{Introduction}

The celebrated dimension reduction result of Johnson and Lindenstrauss \cite{JLlemma} asserts that every $n$--point subset of $\ell_2$ admits a bi--Lipschitz embedding into an $O(\log n)$--dimensional linear subspace of $\ell_2$. Combining Bourgain's embedding \cite{Bourgain-Lipschitz} with Dvoretsky's theorem \cite{Dvoretzky} and the Johnson--Lindenstrauss lemma yields a remarkable generalization beyond the Euclidean setting:

\begin{center}
\vspace{0.5em}
\textit{Any} $n$--point metric space admits an $O(\log n)$--distortion embedding into\\ an $O(\log n)$--dimensional subspace of \textit{any} infinite dimensional normed space. 
\vspace{0.5em}
\end{center}

This foundational result has driven decades of research in metric geometry aimed at quantitatively characterizing which metric spaces admit low--distortion embeddings into low--dimensional normed spaces. Yet, the sharpness of the above statement remains open. This article closes that gap by establishing sharp bounds for the metric dimension reduction modulus.

\begin{definition}[Metric dimension reduction modulus \cite{naor-snapshot}]
    Let $n \in \mathbb{N}$ and $\alpha := \alpha(n) \in[1,\infty)$.
Let $X$ be a normed space, possibly infinite dimensional.
Denote by $k^\alpha_n(X)$ the minimum $k \in\mathbb{N}$
such that every $n$ point subset $S\subset X$ embeds with bi-Lipschitz distortion
at most $\alpha$ into some 
$k$--dimensional
linear subspace $F_S$ of $X$ (with respect to the metric induced by $X$). That is, there is a mapping $f:S\to F_S$ satisfying
\[
    \|f(s)-f(s')\|_X\leq \|s-s'\|_X\leq \alpha \,\|f(s)-f(s')\|_X,
\quad \mbox{ for all }s,s'\in S.
\]
\end{definition}

Since every finite metric space and every finite-dimensional normed space can be isometrically embedded into $\ell_\infty$, the quantity $k^\alpha_n(\ell_\infty)$ is the minimum $k$ such that every $n$--point metric space can be $\alpha$--embedded into some $k$-dimensional normed space $X$. The quantity $k^\alpha_n(X)$, particularly with $X = \ell_\infty$, plays a crucial role in the theory of metric embeddings. Works developing bounds on $k^\alpha_n(\ell_\infty)$ include \cite{Bourgain-Lipschitz,JLS87,deReynaPiazza92,LLR95,Matousek-distortion,ABN2011,Naor2017,naor-snapshot}. We refer to Naor's 2018 ICM plenary lecture \cite[Section 1.3]{naor-snapshot} for a comprehensive overview of the vast literature around metric dimension reduction.

In terms of the dimension reduction modulus, Bourgain's embedding theorem \cite{Bourgain-Lipschitz} implies 
that for $\beta$ sufficiently large, $k_n^{\beta \log n}(\ell_\infty) = O(\log n)$, and more generally
(by combining with the Johnson--Lindenstrauss lemma and
Dvoretsky's theorem \cite{Dvoretzky}), 
for every infinite--dimensional Banach space $X$
and sufficiently large $\beta>0$,
\begin{equation}\label{modulusupperbound} 
k^{\beta\log n}_n(X)=O(\log n).    
\end{equation}
Subsequent works, culminating in \cite{Matousek-distortion,ABN2011}, extended the upper bound $k^{\alpha}_n(\ell_\infty)=O(\log n)$ to
$\alpha = \beta \log n$ for every constant $\beta>0$. In contrast, the strongest
lower bound on $k^\alpha_n(\ell_\infty)$ 
for $\alpha = \Theta(\log n)$ has remained:
\begin{equation}\label{eq:modulus LB}
    k^\alpha_n(\ell_\infty)=\Omega\Big(\frac{\log n}{\log\log\log n}\Big).
\end{equation}
Both \eqref{modulusupperbound} and \eqref{eq:modulus LB} extend to larger distortion values $\alpha$, but with a persisting gap. The current state of knowledge for the dimension reduction modulus with superlogarithmic distortion is: 

\begin{theorem}[State of the art \cite{Matousek-distortion,ABN2011,naor-snapshot}]
    For every fixed $\beta > 0$, there are some positive constants $c$ and $C$, depending only on $\beta$, such that the following holds for $n$ sufficiently large. For any $\alpha$ with $\beta \log n \le \alpha \le n$, 
    \begin{equation}\label{eq:state of the art}
           \frac{c \, \log n}{\log(\frac{\alpha \log \log n}{\log n})} \le  k^{\alpha}_n(\ell_\infty) \le \frac{C\, \log n}{\log(\frac{\alpha}{\log n}+1)}\,.
    \end{equation}
\end{theorem}
See \cref{sec:existing methods} for a sketch of the lower bound. We refer to \cite[formulae (12), (16)]{naor-snapshot} for a discussion of the asymptotics of $k_n^\alpha(\ell_\infty)$ with distortion $\alpha = O(\log n)$. A notable consequence towards our setting of these asymptotics is that $\alpha = \Theta(\log n)$ is the center of a critical window in which $k_n^{\alpha}(\ell_\infty)$ undergoes a rapid phase transition from power--law to sublogarithmic scaling. 

Closing the gap between upper and lower bounds in \eqref{eq:state of the art} has remained a long--standing problem. The case $\alpha = \Theta(\log n)$ is particularly important, appearing as Question 22 in Naor's ICM plenary lecture \cite{naor-snapshot}.

\begin{problem}[Naor \cite{naor-snapshot}]
    Determine the asymptotics of $k^\alpha_n(\ell_\infty)$ in the regime $\alpha=\Theta(\log n)$.
\end{problem}

Several perspectives motivate the particular interest in logarithmic distortion (equivalently, per our main result, logarithmic dimension). Beyond the intrinsic interest of phase transitions, there are the mentioned connections with the scalings present in the Johnson--Lindenstrauss lemma and Bourgain's embedding, especially regarding sharpness of the latter. 
Moreover, $\log n$ is a 
metric analog of ``dimension'' for (generic) $n$--point spaces: this value appears in various notions within the Ribe program, such as doubling dimension \cite[Remark 39]{naor-snapshot}) and sphericity \cite{FranklMaehara,AT2024}. Characterizing the distortion values $\alpha$ for which the bi--Lipschitz metric dimension proxy $k_n^\alpha(\ell_\infty)$ coincides with $\log n$ is thus natural. 

The regime of logarithmic distortion and dimension also has computational significance: algorithm design often uses low--distortion, low--dimensional embeddings of metric spaces into normed spaces where linear structure can be algorithmically exploited. If either dimension or distortion is sub--logarithmic, the other scales polynomially in the worst case. Thus logarithmic dimension and distortion offers a computational ``sweet spot'' in the trade-off between runtime complexity (scaling with dimension) and approximation accuracy (scaling with distortion). \\

Our main result is a sharp characterization of $k_n^\alpha(\ell_\infty)$, up to multiplicative constants, for $\alpha = \Omega(\log n)$. This completely resolves Naor's ICM problem and closes the long--standing gap in \eqref{eq:state of the art}.
\begin{theorem}\label{thm:main}
For any fixed $\beta > 0$ and $n$ sufficiently large, the following holds. For any distortion $\alpha$ satisfying $\beta \log n \leq \alpha\leq n$,
\[
    k^{\alpha}_n(\ell_\infty)=\Theta\bigg(\frac{\log n}{\log\big(\frac{\alpha}{\log n}+1\big)}\bigg).
\]
\end{theorem}

In particular, the upper bound in \eqref{eq:state of the art} is sharp. Our contribution is the matching lower bound.

\subsection{Graph embeddings}
The $n$--point metric space producing the lower bound in \cref{thm:main} is supplied by a random regular graph. Such graphs are common models of metric spaces with strong expansion properties, making them particularly useful for proving non--embedding results. Since metric embeddings of graphs is an important research direction in its own right, we offer a reformulation (and strengthening) of our main result:

\begin{theorem}[Bi--Lipschitz dimension reduction for graphs]\label{thm:graphs}
    For any fixed $\beta > 0$, $\Delta\geq 3$ and $n$ sufficiently large, so that $\Delta n$ is even, the following holds. Let $\alpha$ satisfy $\beta \log n \leq \alpha\leq n$ and let $G_n$ be a uniformly random $\Delta$--regular graph on $[n]$, conditioned on the event of being connected. Denote the set of normed spaces with dimension at most $d$ by $\mathcal{B}(d)$ and let $k := k_n^\alpha(\ell_\infty)$. Then, for some positive constants $c$, $C$, and $\ctail$ (only allowed to depend on $\beta$ and $\Delta$), 
    \begin{equation}\label{eq:graph embedding LB}
        \PP{\exists \,X \in \mathcal{B}(c\,k)\,:\, G_n ~\text{embeds into $X$ with distortion at most $\alpha$}} \le e^{-\ctail\,n \log n}\,,
    \end{equation}
    and 
    \begin{equation}\label{eq:graph embedding UB}
        \PP{\exists \,X \in \mathcal{B}(C\,k)\,:\, G_n ~\text{embeds into $X$ with distortion at most $\alpha$}} = 1\,.
    \end{equation}
\end{theorem}

\begin{remark}
    Consider the mechanisms that allow a graph to admit atypically low--dimensional embeddings. Even exponentially rare fluctuations in diameter, spectral gap, small subgraph counts, etc., cannot be mechanisms for atypically low--dimensional structure in graphs, by virtue of the super--exponential tail bound in \cref{thm:graphs}. 
\end{remark}

Beginning with the landmark work of Linial, London, and Rabinovich \cite{LLR95}, bi--Lipschitz embeddings of graphs into normed spaces have been actively studied in both deterministic and randomized settings, leading to remarkable algorithmic developments. Two seminal examples include algorithms for sparsest cut \cite{sparsestcut} and approximate nearest neighbor search \cite{ANNS-1,ANNS-2}, both generating enormous bodies of follow--up work. We refer to the classical surveys \cite{indyk-survey, indyk-survey2} and \cite{mat-book} for algorithmic and theoretical aspects of graph embeddings, respectively, and to \cite{Eskenazis2022} for modern developments. 

Various other notions of embeddings for graphs have also been studied, often motivated by algorithmic tasks such as clustering. Particularly relevant to the current article is \textit{geometric} embedding \cite{AT2024,ADTT2024}, which demands that vertices of a graph b are embedded at distance less than one if and only if they are adjacent in the graph. A closely analogous statement to \eqref{eq:graph embedding LB}, specialized to $\alpha = \Theta(\log n)$, was shown in \cite{AT2024}: with probability $1 - e^{-\Omega(n \log n)}$, a random three--regular graph does not admit a geometric embedding into any normed space of sub--logarithmic dimension.

\subsection{Technical overview}
Since the upper bound in \eqref{eq:state of the art} contains \eqref{eq:graph embedding UB}, establishing \eqref{eq:graph embedding LB} suffices for proving \cref{thm:graphs} and thus \cref{thm:main}. We begin by surveying existing methods. For simplicity, consider only the regime $\alpha = \Theta(\log n)$.

\subsubsection{Existing methods} \label{sec:existing methods}
A naive volumetric argument proceeds as follows. The diameter of a typical random $3$--regular graph $G$ is $\Theta(\log n)$. Assume there is a mapping $f$ from $G$ to a $d$--dimensional normed space $X$ incurring distortion $\alpha$; we seek to bound $d$ in terms of $\alpha$. By linearity of $X$ and scale--invariance of distortion, we may assume $f$ is an \textit{expansion} map (i.e., $d_G(x,y) \le \|f(x)-f(y)\|_X$ for all $x,y$). Then the diameter of $f$ is at most of order $\alpha \log(n)$ in $X$. By $3$--regularity of $G$, no open $X$--ball of radius $2$ centered at some $f(x)$, can contain more than four points of the image of $f$. Duality between packing and covering implies the existence of at least $n/4$ disjoint $X$--balls of radius $1$ within the $X$--ball of radius $\alpha \log(n)+1$. Volumetric considerations imply:
\[
    \frac{n}{4} 
    \leq (1+\alpha \log n)^d\,, ~ \text{ and thus } d \gtrsim \frac{\log n}{\log( \alpha \log n)} \approx \frac{\log n}{\log \log n}\,.
\]
That is, embeddings of a random graph must satisfy conflicting confining and repulsive constraints, due to diameter and $3$--regularity respectively, that can only be balanced in low dimensions. The optimal lower bound turns out to be $\log n$; our task is to remove the $\log \log n$ deficit.

The theory of nonlinear spectral calculus offers significant improvement. The confinement of the image of $f$ to a ball of radius $\alpha \log n$, which is the source of the $\log\log n$ deficit, comes from a \textit{worst case} bound on distances between embedded vertices. For specific choices of normed space $X$, random graphs are known to satisfy a nonlinear Poincar\'e inequality, which implies the \textit{average case} bound that most vertices are only at distance $\alpha$ from each other. This improvement yields the optimal lower bound $d \gtrsim \log n$ for these structured spaces $X$, as well as $d \gtrsim \log n/\log\log\log n$ for generic spaces $X$. This is exactly the lower bound on $k_n^\alpha(\ell_\infty)$ in \eqref{eq:modulus LB} and \eqref{eq:state of the art}. We refer to \cite{Naor2017} and \cite[Section 2]{AT2024}\footnote{A different notion of embedding is considered here, but obvious changes extend the discussion to bi--Lipschitz embeddings.} for definitions and further exposition.

\subsubsection{Overall approach: the union bound} We view a $d$--dimensional normed space as $(\RR^d,\|\cdot\|_X)$, i.e., as $\mathbb{R}^d$ with a different unit ball. Our approach is a union bound over all possible (expansive) $\alpha$--embeddings of $\Delta$--regular graphs into $\RR^d$, for each possible norm. More precisely, our main task is the construction of a net $\mathcal{N}$ with the following properties:
\begin{itemize}
    \item The net $\net \subset (\RR^k)^n$ has cardinality at most $e^{n \log n}$. 
    \item Each potential embedding $\vx = (x_1,\dots,x_n) \in (\RR^k)^n$ of a $\Delta$--regular graph into some $k$--dimensional normed space $X$ is close (distortion less than two) to some vector $\hx \in \mathcal{N}$.
    \item Each $\hx \in \mathcal{N}$ is a distortion  $2\alpha$--embedding for less than an $e^{-1.1\,n \log n}$--fraction of all $\Delta$--regular graphs on $n$ vertices. 
\end{itemize}
The existence of such a net will imply the result.

\subsubsection{Net of the Banach--Mazur compactum}
The set of normed spaces can be equipped with a natural metric in terms of distortion, resulting in the so--called Banach--Mazur compactum (defined in the next section). With this metric, it becomes well--defined to seek a net on the Banach--Mazur compactum. By the standard equivalences between normed spaces, unit balls, and symmetric convex bodies, constructing a net for the Banach--Mazur compactum is equivalent to finding low--complexity approximations of convex bodies. This is an extremely well--studied task in computational and convex geometry. An existing result of Pisier \cite{pisier} directly supplies the desired net. 

\subsubsection{Sparse tuples}

We introduce the key notion of $\lambda$--sparsity: an $n$--tuple of points $\vx \in (\RR^d,\|\cdot\|_X)^n$ is $\lambda$--sparse if fewer than $n^{\eps}$ points lie in any ball of radius $\lambda$, for some small universal constant $\eps > 0$. By volumetric considerations, tuples which are \textit{not} $\lambda$--sparse cannot correspond to $\alpha$--embeddings of a $\Delta$--regular graph, for appropriate choices of $\lambda$ in terms of $(\alpha,d)$. The advantage of restricting to sparse tuples is: if $\vx = (x_1,\dots,x_n)$ is both $\alpha$--sparse and an $\alpha$--embedding of a graph $G = ([n],E_G)$, then for each $i$, 
\begin{equation}\label{eq:edge - order statistics}
    E_G \subset \{\{i,j\}\,:\, x_j \text{ is among the $n^\eps$--closest points to $x_i$}\}.
\end{equation}

\subsubsection{Adaptive net}

For a fixed $d$--dimensional normed space $X$, it remains to construct a low--cardinality net of all possible embeddings of $\Delta$--regular graphs. Specifically, we need a net on the set of sparse tuples. For $d = \Theta(\log n)$, a minimal net of all $n$--tuples in $d$--dimensions---say with precision $r$, and restricted to the diameter $D$ ball---has cardinality $\exp\{\Omega(n \log(n)\, \log(D/r))\}$. Such a net is too large for $D/r$ diverging, and too coarse to be useful for $D/r$ fixed. 

We remedy this issue by using an efficient \textit{adaptive} net $\net$. A similar net was used in \cite{AT2024} within the context of geometric embeddings of random graphs. We sketch the key ideas for constructing $\net$ (and the associated ``discretization'' function $f:X^n \to \net$):
\begin{enumerate}
    \item In view of \eqref{eq:edge - order statistics}, we only require $f$ to preserve: for each $i$, the labels of the $n^\eps$ closest points $x_j$ to $x_i$.  
    \item Instead of taking a standard net, which is intractable, we build an \textit{adaptive} net. Consider the toy problem of discretizing a \textit{fixed} $\vx$. Since we only seek to preserve order statistics of distances, $\net$ can be coarse in regions of $X$ containing few points of $\vx$.
    \item To quantify where $\vx$ is sparse or dense in $X$, randomly select indices $\seeds \subset [n]$. The corresponding points $(x_s)_{s \in \seeds}$ are called \textit{seeds}. The distribution of seeds naturally inherits the geometry of $\vx$: the expected number of seeds in a region of $X$ is proportional to the density of $\vx$. 
    \item The key gain is that the number of seeds can be tiny, $|\seeds| = n^{1-\eps} \ll n$, and $\seeds$ will still capture the geometry of $\vx$. For the toy problem of discretizing this particular $\vx$, take $\net$ as a union of small, standard nets centered around each seed, with diameter and precision proportional to the local density of $\vx$ around that seed. The resulting patchwork of local nets provides the desired construction: $|\net|$ is negligible compared to our goal, and projection of $\vx$ onto $\net$ preserves order statistics of distances. 
    \item Moving from the toy problem to the full construction, $\net$ obviously cannot depend on the vector $\vx$ we seek to discretize. Instead, $\net$ is constructed by first taking a standard net on $X^{|\seeds|}$ for $\seeds$ and then appending ``local nets'' (of all possible diameters and precisions, along a dyadic sequence) around each possible seed location. 
\end{enumerate}

\subsection{Asymptotic notation}
All asymptotics in this paper are with respect to $n$ (the size of the metric space) diverging.
We write $f=O(g)$, $f=o(g)$
whenever $f$ remains bounded with respect to $g$ (respectively, vanishes) as $n\to \infty$. Further, $f=\Omega(g)$ whenever $g=O(f)$, and $f=\Theta(g)$
if simultaneously $f=O(g)$ and $f=\Omega(g)$.

\bigskip

{\bf Acknowledgment.} 
We thank Assaf Naor for the suggestion the application of the
discretization scheme from \cite{AT2024} to
the study of bi--Lipschitz embeddings, and in particular towards estimating $k^\alpha_n(\ell_\infty)$. The second named author is partially supported by NSF grant DMS 2331037.

\section{Preliminaries}

\subsection{Useful facts on random graphs}

\begin{definition}
Given $\Delta\geq 3$, denote by $\dreg$
the set of simple $\Delta$--regular graphs
on $[n]$.
We further write $G\sim \dreg$ for a random graph uniformly
distributed on $\dreg$.
\end{definition}

\begin{definition}
Let $G(n,m)$ be the collection of all simple graphs on $[n]$
with $m$ edges. We write
$G \sim G(n,m)$ for a random graph uniformly
distributed on $G(n,m)$.
\end{definition}

\begin{lemma}[\cite{wormald}]\label{lemma:connectivity}    Let $G$ be drawn uniformly from $\dreg$ for any $\Delta := \Delta(n) \ge 3$ with $\Delta n /2$ an integer.  Then the
graph $G$ is connected asymptotically almost surely. 
\end{lemma}

\begin{lemma}[\cite{enumerate}]\label{lemma:enumerate}
    Let $\Delta= n^{o(1)}$ and $\Delta n /2$ be an integer. Then
    \[
        |\dreg| = (1 + o(1))\,e^{1 - \frac{\Delta^2}{4}} \frac{(\Delta n)!}{(\Delta n /2 )!\,2^{\Delta n /2}(\Delta !)^n}.
    \]
\end{lemma}

\begin{corollary}\label{cor:contiguity}
    Let $\Delta n /2$ be an integer, and let $G$ be drawn uniformly from
    $G(n,\Delta n / 2)$. If $\Delta = n^{\ao{1}}$,
    \[
        \PP{G \in \dreg} = e^{-o(n \Delta \log n)}\,.
    \]
\end{corollary}

We refer, for example, to \cite{AT2024} for a proof of the
above corollary.

\subsection{The Banach--Mazur compactum}

Given a positive integer $d$, denote by $\mathcal{B}_d$
the collection of all $d$--dimensional normed spaces 
(identified up to isometry), endowed with a submultiplicative
metric
$$
d_{BM}(X,Y):=\inf\limits_{T:X\to Y}\|T\|_{X\to Y}
\,\|T^{-1}\|_{Y\to X},\quad X,Y\in \mathcal{B}_d,
$$
where the infimum is taken over all invertible linear operators $T$,
and where $\|T\|_{X\to Y}$ and $\|T^{-1}\|_{Y\to X}$ are standard operator
norms for $T:X\to Y$ and $T^{-1}:Y\to X$, respectively.
We refer to
\cite{NTJ} for a comprehensive discussion of the concept
and its fundamental role in the study of finite-dimensional
normed spaces.
In this note, we will use the following 
metric entropy estimate on the Banach--Mazur compactum: 
\begin{lemma}[\cite{pisier}]\label{lemma:BM-entropy}
    There is some positive universal constant $C_{\text{BM}}$ (satisfying $C_{\text{BM}} \le 10$) and a $2$-net $\netB(d)$ in $\mathcal{B}_d$ such that:
    \[
        |\netB(d)| \le \cE{e^{d\,C_{\text{BM}}}}\,.
    \]
\end{lemma}

\begin{remark}
Note that, from the above lemma,
it follows that as long as the dimension $d$ satisfies $d\leq \frac{\log n}{2\,C_{\text{BM}}}$, there is a $2$--net $\netB(d)$ on the Banach-Mazur compactum $\mathcal{B}_d$ of size at most $\exp(\sqrt{n})$.
\end{remark}

\section{Construction of the discretization scheme}\label{sec:discretization}

The goal of this section is to produce a discretization $\hx:X^n \to \net$, where $X$ is a given finite-dimensional normed space and $\net$ is a discrete subset of the Cartesian product $X^n$ having both small cardinality and ``sufficient density.'' The construction presented below can be viewed as a  version of the classical net argument with advanced features provided through multiscaling and seeding. 
As we noted in the introduction, a related discretization scheme
first appeared in \cite{AT2024} to deal with {\it geometrical} embeddings of random graphs. Whereas the discretization presented here is not identical to the one in \cite{AT2024}, many similarities exist. For that reason, proofs of certain claims closely matching corresponding statements from \cite{AT2024} are moved to the appendix.

\medskip

{\bf Notation.}
For the remainder of this article, $\eps$ and $c_0$ are some fixed parameters, independent of $X$ and $n$, that are assumed to be sufficiently small. That is, $\eps \le \eps'$ and $c_0 \le c_{0}'$ for some positive universal constants $\eps'$ and $c_0'$ which do not depend on $X$ or $n$. Without making effort to optimize, the following choices suffice:
\begin{align*}
    \eps' &=  .1  \\
    c_0' &= .01 \\
    \ctail &= \eps/20
\end{align*}
We also require an integer parameter $L\in [1,n^{1-\eps}]$, which will be fixed in \cref{thm:bilipschitz}. As $L$ and $X$ will only be varied in \cref{sec:bilipschitz}, one can view $L$ and $X$ as fixed parameters in \cref{sec:discretization};  dependencies on $L$ and $X$ will be suppressed wherever there is no ambiguity in order to lighten notation. 
Below, ``$n$ sufficiently large'' means $n \ge n_0$ for some function $n_0 := n_0(\eps, c_0)$. We emphasize that $n_0$ does not depend on the underlying normed space $X$.

A $d$--dimensional normed space $X$ will be always be viewed as the space $\RR^d$ endowed with a norm $\|\cdot\|_X$. The dimension $d$ of the space will always be assumed to satisfy
\begin{equation}\label{dcondition}
1\leq d\leq n^{\eps/4}.    
\end{equation}
The unit ball in $X$ will be denoted by $B_X$. More generally, a ball of radius $r$ centered at a point $y\in X$ will be denoted by $B_X(y,r)$.
Given a point $y\in X$ and a non-empty closed subset $A$ of $X$,
the $\|\cdot\|_X$--projection of $y$ onto $A$ is any point in $A$ with the minimal $\|\cdot\|_X$--distance to $y$.
Whenever projections of points onto subsets are not uniquely defined, a representative will be chosen arbitrarily. In what follows, $\vx$ will always denote an $n$-tuple of points $(x_1,\dots,x_n) \in (\RR^d)^n$. Occasionally, we will view $\vx$ as a multiset (i.e disregard the ordering).

\bigskip

\begin{definition}[Sparse $n$--tuples]\label{def:sparse}
    Let $\|\cdot\|_X$ be any norm in $\RR^d$, and let $\lambda>0$ be any parameter (possibly depending on $n$). Say
    that a tuple of points $\vx$ is {\it $\lambda$-sparse} if, for all $i\in[n]$, $\|x_i - x_j\|_X\le \lambda$ for less than $n^{\varepsilon}$ values of $j \in [n]\setminus \cb{i}$.
\end{definition}

As a rough interpretation of the above notion, an $n$--tuple is $\lambda$-sparse if for a vast majority of pairs of vectors, the distance between the vectors is greater than $\lambda$. In what follows, it will be crucial to us that the property is essentially preserved even if the threshold $\lambda$ is multiplied by a constant. Specifically, the parameter $L$ mentioned in passing above, is defined later in such a way that in any ball of radius $72\lambda$, there are at most $Ln^{\eps}$ other vectors (see 
\cref{def:scale} and \cref{lemma:scale-lb} below).

\begin{definition}[Domain]
Assume that $X$ is a normed space, and let $D,\lambda>0$ be parameters
satisfying
\begin{equation}\label{aksdjfnalkfnflks}
n^2\geq D\geq 2\lambda\geq 1.
\end{equation}
Let the domain $\tuples(D,\lambda)=\tuples(D,\lambda,X)$ of the discretization function be given by
\begin{equation}\label{eq:domain}
    \tuples(D,\lambda)=\tuples(D,\lambda,X) := \cb{\vx \in (B_{X}(0;\, D))^n\,:\, \vx \text{ is $\lambda$-sparse in $X$}}\,.
\end{equation}
\end{definition}

The discretization of $X$ will be based on a multiscale net-argument with
seeding.
Recall that an {\it $r$-net} on a set $A \subset \RR^{d}$ in $X=(\RR^{d},\|\cdot\|_X)$  is a collection of points $(y_i)_i$ so that for all $x \in A$, there exists $y_i$ with $\|x - y_i\|_X \le r$. We will always assume that the points $y_i$ are contained in $A$ for convenience. We introduce the following notation: 
\begin{definition}[Simple nets]
    Given $r>0$ and $y\in X$, define $\net(y, r; c_0r)$ as an (arbitrary) $c_0r$-net in $B_X(y; r)$ of minimum cardinality. Further, given a parameter $D>0$, let $\net_0=\net_0(D)$ denote an arbitrary $1$-net in $B_X(0; D)$ of minimum cardinality.
\end{definition}

\begin{definition}[Multiscale net]\label{def:mulitscale-net}
    For every integer $\ell$ we use the shorthand 
    \[
        r_\ell := 2^{\ell}\,.
    \]
    Let $X$ be a normed space, and let $D,\lambda>0$ be parameters satisfying \eqref{aksdjfnalkfnflks}. Define the {\it multiscale net} $\net=\net(D,\lambda)$ as:
    \[
        \net := \net_0(D) \cup \pa{ \bigcup_{y \in \net_0(D)}\; \bigcup_{\ell = \lfloor \log_2\lambda\rfloor+1}^{\lceil\log_2 (2D)\rceil} \net(y,\, r_\ell; c_0 r_\ell) } \,.
    \]
\end{definition}
In what follows, we will use points from $\net$ to approximate locations
of vertices in a graph embedding into $X$. A key technical idea is that it will suffice to roughly maintain the list of closest neighbors to each vertex under this approximate embedding, rather than attempting to maintain the actual pairwise distances between all vertices. This is a significant source of efficiency in the argument. Towards making this precise, we introduce the following definition.

\begin{definition}[Scale of separation]\label{def:scale}
    Given $\vx = (x_1,\dots,x_n) \in (\RR^{d})^n$ and a parameter
    $L\in [1,n^{1-\eps}]$,
    define the \textit{local scale of separation for $x_i$} by:
    \[
        \ell_i := \ell_i(\vx,X) = \min\cb{\ell \in \mathbb{Z}\,:\, \ba{\cb{j\in [n]\,:\, \|x_i-x_j\|_X \le r_\ell }} \ge L\,n^{\eps}}\,, \quad 
        1\leq i\leq n\,.
    \]
    Informally, for every $i$, the local scale of separation for $x_i$ is the logarithm of the radius of a ball in $X$ centered at $x_i$ containing about $n^{\eps}$ points $x_j$. We further write $\ell(\vx)=\ell(\vx,X)$ for the tuple $(\ell_i)_{i=1}^n$.
\end{definition}

\begin{remark}\label{boundsonell}
Note that whenever $\vx$ is $\lambda$-sparse then necessarily $2^{\ell_i}> \lambda$ for all $i \in [n]$.
Moreover, whenever $\vx\in (B_{X}(0,r))^n$ for some $r>0$, we must have $\ell_i\leq \lceil\log_2(2r)\rceil$ for all $i$.
In particular, for every $\vx=(x_i)_{i=1}^n\in\tuples(D,\lambda)$
    we have $\ell_i\in \{\lfloor \log_2\lambda\rfloor+1,\dots, \lceil\log_2(2D)\rceil\}$
    for all $i\leq n$.
\end{remark}

In the next proposition we define {\it seeds}, which are
small-cardinality subsets of $\vx$ that play a key role in the discretization scheme. A tuple $\vx$ will be embedded into simple (local) nets centered at the seeds.

\begin{proposition}[Existence of good seeds]\label{prop:seeds}
    Let $n$ be sufficiently large (independent of $X$). For any $n$--tuple $\vx$, and for each integer $\ell\geq 0$, there exists a multiset $\seeds_\ell := \seeds_\ell(\vx) \subset \vx$, with $|S_\ell| =\lfloor n^{1-\eps}\log^2 n\rfloor$ and the following property. For any component $x_i$ of $\vx$ with $\ell_i = \ell$, there exists $y \in \seeds_\ell$ with 
    \[
        \|y - x_i \|_X \le r_{\ell_i}\,.
    \]
\end{proposition}
\begin{proof}
    Fix $\vx$ and $\ell$. Let $\seeds_\ell$ be the collection of $\lfloor n^{1-\eps}\log^2 n\rfloor$ points selected independently uniformly at random from $\vx$ (viewed as a multiset). Let $z \sim \mathrm{unif}(\vx)$. For all $x_i \in \vx$, we have:   
    \begin{align*}
        \PP{\bigcap_{y \in S_\ell} \cb{\|y - x_i \|_X > r_\ell}} = \PP{\|z - x_i \|_X > r_\ell}^{\lfloor n^{1-\eps}\log^2 n\rfloor} \le \pa{1 - \frac{n^{\eps}}{n}}^{\lfloor n^{1-\eps}\log^2 n\rfloor} &= n^{-\omega(1)} \,.
    \end{align*}
    Taking a union bound over the $n$ possible values of $i$, it easily follows that with high probability (in particular, positive probability), our random selection of $\seeds_\ell$ has the desired property. Thus, by the probabilistic method, there must deterministically be some satisfactory realization of $\seeds_\ell$.
\end{proof}

\begin{definition}[Discretization]
    Let $D,\lambda>0$ be parameters satisfying \eqref{aksdjfnalkfnflks}, and let $L\in [1,n^{1-\eps}]$. The \textit{discretization scheme} is the map \[
        \hx := \hx(\vx): \tuples(D,\lambda,X) \to \net(D,\lambda)^n\,,
    \]
    constructed as follows. For each $\ell \in \{\lfloor \log_2\lambda\rfloor+1,\dots, \lceil\log_2(2D)\rceil\}$, construct $S_\ell$, as defined in \cref{prop:seeds}. Let $\hat \seeds_{\ell}$ denote a multiset which is the $\|\cdot\|_{X}$-projection of $\seeds_\ell$ onto $\net_0$. 
    Then, for each $i \in [n]$:
    \begin{enumerate}
        \item Let $\hat s_i$ be the $\|\cdot\|_{X}$-projection of $x_i$ onto $\hat \seeds_{\ell_i}$. 
        \item Let $\hat x_i$ be the $\|\cdot\|_{X}$-projection of $x_i$ onto $\net(\hat s_i, r_{\ell_i}; c_0 r_{\ell_i}) \subset \net$. 
    \end{enumerate}
    Define $\hx := (\hat x_1,\dots, \hat x_n) \in \net^{n}$. 
\end{definition}

\begin{remark}\label{distxitohatxi}
By definition of $S_{\ell_i}$, the distance
from $x_i$ to $S_{\ell_i}$ is at most $r_{\ell_i}$.
Since $\net_0$ is a $1$--net, the distance from $\hat \seeds_{\ell_i}$
to $x_i$ is at most $r_{\ell_i}+1$. Thus,
$\|x_i-\hat s_i\|_X\leq r_{\ell_i}+1$, $i\leq n$.
In turn, this implies $\|x_i-\hat x_i\|_X
\leq c_0 r_{\ell_i}+1$ for all $i\leq n$.
\end{remark}

Note that the above construction 
depends on the underlying normed space $X$. We collect some properties of this discretization.  

\begin{lemma}\label{lemma:discretization} Let $n$ be sufficiently large (independently of $X$), let $d$ satisfy \eqref{dcondition}, let the parameters $D,\lambda$ satisfy \eqref{aksdjfnalkfnflks}, and let $L\in [1,n^{1-\eps}]$.
Then
    \begin{enumerate}
        \item (Cardinality estimate)
        \[
            \ba{\cb{(\hx(\vx),\, \ell(\vx))\,:\, \vx \in \tuples(D,\lambda,X) }} \le 
    \big(\pa{3/c_0}^d\big)^n\cdot \big(n^{1-\eps}\log^4 n\big)^n
    \cdot n^{n^{1-\eps/2}}.
        \]

        \item (Upper bound on expansion of distances)
        Let $\vx \in \tuples(D,\lambda,X)$ and $\hx := \hx(\vx)$. For all $t>0$ and all $(i,j) \in [n]^2$, 
        \[
            \|\hat x_i - \hat x_j\|_X \le (1+2c_0)\|x_i-x_j\|_X +3c_0 r_{\ell_i} + 2.
        \]

        \item (Preservation of local scales) Let $\vx \in \tuples(D,\lambda,X)$ and $\hx := \hx(\vx)$. For all $i \in [n]$, 
        \[
            \ba{\cb{j \in [n]\,:\, \|\hat x_i - \hat x_j\|_X \le \frac{4}{9}\,r_{\ell_{i}}-2}} \le L\,n^{\eps}\,.
        \]
    \end{enumerate}
\end{lemma}

In view of strong similarities between the proof of the above lemma
and of the corresponding statement in \cite{AT2024},
the proof of Lemma~\ref{lemma:discretization} is presented in the appendix.

\begin{lemma}[Lower-bound on scales of separation]\label{lemma:scale-lb}
Let $D,\lambda>0$ be parameters satisfying \eqref{aksdjfnalkfnflks},
let $L\in[1,n^{1-\eps}]$,
and assume additionally that
$$
300^d\leq \frac{L}{2}.
$$
Then for all $\vx \in \tuples(D,\lambda,X)$ and all $i \in [n]$,
\[
    \ba{\cb{j \in [n] \,:\, \|x_i - x_j\|_X \le 72\lambda}  } < L\,n^{\eps}\,.
\]
In particular, $r_{\ell_i} \geq 72\lambda$ for all $i \in [n]$. 
\end{lemma}

\begin{remark}
The constant $72$ in the lower bound for $r_{\ell_i}$ can be replaced with arbitrary fixed number, by adjusting the assumption on $d$ accordingly.    
\end{remark}

The proof of the above lemma closely follows an argument from \cite{AT2024};
we provide the complete proof for an interested reader in the appendix.

\section{Bi-Lipschitz non-embedding}\label{sec:bilipschitz}

The goal of this section is to complete the proof of the main
result of the note: for any $\beta > 0$, with high probability,
a random $\Delta$--regular
graph on $n$ vertices does not embed into
any normed space of dimension $d$ at most $c\log n$
with bi--Lipschitz distortion less than $\beta \exp(\frac{\log n}{Cd})\log n$,
for an appropriate choice of constants $c,C>0$.
By linearity of any normed space metric, 
we can restrict our attention to embeddings which are {\it expansions},
i.e do not decrease pairwise distances.
Given a normed space $X$, an $n$--tuple $\vx$ of elements of $X$,
a graph $G$ on $[n]$,
and a parameter $\alpha$, we will write $G \stackrel{\distortion}{\hookrightarrow} (X,\vx)$ whenever
the mapping $i\to x_i$, $i\leq n$, satisfies
\begin{equation}\label{bfGembeds}
d_G(i,j)\leq \|x_i-x_j\|_X
\leq \alpha\,d_G(i,j),\quad i,j\in[n],
\end{equation}
and, similarly,
$G \not \embeds (X,\vx)$ if \eqref{bfGembeds} does not hold.
By convention, we will assume that whenever the graph $G$
is not connected, \eqref{bfGembeds} fails for all choices
of $X,\alpha$, and $\vx=(x_i)_{i=1}^n$.

The proof of the main result is split into two
cases: either the mapping of our random graph into a normed space
$X$ is ``clustered''
in a sense that the image contains a significant group of points
at relatively small distance from each other, or it is {\it sparse}
in which case we apply the discretization scheme from the previous
section.

In what follows, given a dimension parameter $d$,
we define the target distortion $\alpha$ and the
threshold parameter $\lambda$ for sparsity as
\begin{equation}\label{alphadefinition}
\alpha:= C_{\text{\tiny{\ref{sec:bilipschitz}}}}^{-1}\exp\left(\frac{\log n}
{C_{\text{\tiny{\ref{sec:bilipschitz}}}}\,d}\right)
\log n,\quad \lambda:=\alpha,
\end{equation}
where $C_{\text{\tiny{\ref{sec:bilipschitz}}}}\geq \max(1000, \beta^{-1})$
is a sufficiently large constant which is allowed to depend on $\Delta$.

\subsection{Non-embedding into a non--sparse tuple}

Here, we consider the setting 
where the image of the embedding of $G$ into a normed space $X'$
contains a ``cluster'' of at least $n^{\varepsilon}$
points at distance $O(\lambda)$ from each other.
In that case, we can show {\it deterministically}
that the mapping has bi--Lipschitz distortion greater than $\alpha$
(with $\alpha$ defined above):
\begin{proposition}\label{prop:non-sparseI}
Let $\Delta\geq 3$, let $n\geq n_\Delta$
be a sufficiently large integer
such that $\Delta n$ is even, and let $1\leq d\leq \log n$. Let $G$ be a $\Delta$--regular
connected graph on $[n]$, let $X'$
be a $d$--dimensional normed space, and $\vx$
be an $n$--tuple in $X'$
such that
for some $i\leq n$,
\begin{equation}\label{aldhbfoasfjhbsljf}
\big|\big\{j\neq i:\;\|x_j-x_i\|_{X'}\leq 2\lambda\big\}\big|
\geq n^{\varepsilon}.
\end{equation}
Then $G \not \embeds (X',\vx)$.
\end{proposition}
\begin{proof}
We will assume that $n$ is sufficiently large so that
$3(\Delta-1)^{\frac{\varepsilon}{2}\log_\Delta n}\leq n^{\varepsilon/2}$.
Denote the set of indices $j$ from \eqref{aldhbfoasfjhbsljf} by $J$,
so that $|J|\geq n^{\varepsilon}$.

First, we claim that the assumptions imply
that there is a point $x_{i'}\in[n]$
with at least $n^{\varepsilon/2}$
points $x_j$, $j\neq i'$, at distance at most
$\frac{\varepsilon}{2}\log_\Delta n$ from $x_{i'}$.
Indeed, assume the opposite.
We get that
for every point $z\in X'$,
there are at most $n^{\varepsilon/2}+1$ indices $j\leq n$
with $\|z-x_j\|_{X'}\leq\frac{\varepsilon}{4}\log_\Delta n$.
Therefore, the union of translates
$$x_j+B_X\Big(0,\frac{\varepsilon}{4}\log_\Delta n\Big),\quad j\in J,$$
covers every point in
$x_i+B_X(0,2\lambda+\frac{\varepsilon}{4}\log_\Delta n)$ at most 
$n^{\varepsilon/2}+1$ times, and each of those translates is 
entirely contained in $x_i+B_X(0,2\lambda+\frac{\varepsilon}{4}\log_\Delta n))$.
A standard volumetric argument implies
$$
\Vol\Big(B_X\Big(0,2\lambda+\frac{\varepsilon}{4}\log_\Delta n\Big)\Big)
\geq n^\varepsilon\big(n^{\varepsilon/2}+1\big)^{-1}
\Vol\Big(B_X\Big(0,\frac{\varepsilon}{4}\log_\Delta n\Big)\Big),
$$
that is,
$$
\Big(2\lambda+\frac{\varepsilon}{4}\log_\Delta n\Big)^d
\geq \frac{n^\eps}{n^{\varepsilon/2}+1}
\Big(\frac{\varepsilon}{4}\log_\Delta n\Big)^d.
$$
The last inequality is clearly false, assuming that $d\leq\log n$
and that the constant $C_{\text{\tiny{\ref{sec:bilipschitz}}}}$
from the definition of $\lambda$ is sufficiently large.
The contradiction shows that the original claim is true.

By $\Delta$--regularity, for every integer $k\geq 1$ there are less than
$3(\Delta-1)^{k}$ vertices $j\neq i'$ in $G$
at distance at most $k$ from $i'$.
In particular, there are less than
$$
3(\Delta-1)^{\frac{\varepsilon}{2}\log_\Delta n}\leq n^{\varepsilon/2}
$$
points at graph distance at most $\frac{\varepsilon}{2}\log_\Delta n$ from $i$.
Combining the last assertion with the above claim,
we get a point $j\neq i'$ in $G$ such that
$d_G(i',j)>\frac{\varepsilon}{2}\log_\Delta n$
and $\|x_j-x_{i'}\|_{X'}\leq \frac{\varepsilon}{2}\log_\Delta n$.
The result follows.
\end{proof}

\subsection{Robust non-embedding into a $\lambda$--sparse tuple}

The next proposition provides a non-embedding result
for our random graph model conditioned
on the requirement that the image of the embedding is
$\lambda$--sparse.

\begin{proposition}[Robust non-embedding into $\lambda$--sparse tuple]\label{prop:non-embedding}
Let $\Delta\geq 3$, let $n$ be a sufficiently
large integer with $\Delta n$ even,
and let dimension parameter $d\geq 1$ satisfy $300^d\leq n^\varepsilon/2$. 
Let $\alpha$ and $\lambda$ be defined by \eqref{alphadefinition}.
Let $X$ be any $d$--dimensional normed space,
and let $B_{\text{BM}}(X;\,2)$ denote the set of $d$-dimensional normed spaces having (multiplicative) Banach--Mazur distance at most two from $X$.
Let ${\bf G}$ be a uniform random $\Delta$--regular graph on $[n]$.
Then, for some universal constant $c_{\text{\tiny{\ref{prop:non-embedding}}}} > 0$,
    \[
        \PP[{\bf G}]{\exists\, (X',\vx) \in B_{\text{BM}}(X;\,2) \times \domain(n^2,\lambda,X) \,:\, {\bf G} \embeds (X',\vx)}
        \leq \exp(-c_{\text{\tiny{\ref{prop:non-embedding}}}}\,n\log n).
    \]
\end{proposition}

We recall for the reader's convenience that $\domain(n^2,\lambda, X)$ is the set of $\lambda$--sparse $n$--tuples in $B_X(0,n^2)$. 

\begin{proof}
Define
$L:=n^\varepsilon$.
Fix any $d$--dimensional normed space $X=(\RR^d,\|\cdot\|_X)$.
    For $\vx \in \domain(n^2,\lambda,X)$, let $\calL(\vx,X)$ be the set of ``long distances'' for $\vx$, defined as:
    \[
        \calL(\vx,X) := \cb{\{i,j\}\subset [n]\,:\,\|\hat x_i - \hat x_j\|_X > \frac{1}{3} \, \min(r_{\ell_i},\, r_{\ell_j})}
    \]  
    (where the discretization $\hx$ of $\vx$ is also constructed in the space $X$).
    Observe that $\calL(\vx,X)$ depends on $\vx$ only through $\hx$ and $\ell(\vx)$. Thus, the first assertion of \cref{lemma:discretization}
    yields:
    \begin{align}
        \ba{\cb{\calL(\vx,X) \,:\, \vx \in \domain(n^2,\lambda,X)}} &\le  \ba{\cb{(\hx,\ell(\vx)) \,:\, \vx \in \domain(n^2,\lambda,X)}} \nonumber
        \\
        &\le \big(\pa{3/c_0}^d\big)^n\cdot \big(n^{1-\eps}\log^4 n\big)^n \cdot n^{n^{1-\eps/2}} \,. \label{eq:cardinality-long}
    \end{align}
    Fix for a moment any $\vx\in\domain(n^2,\lambda,X)$.
    We claim that whenever the edge-set $E(G)$
    of a graph $G$ on $[n]$ has a non-empty intersection with $\calL(\vx,X)$, then for every $X' \in B_{\mathrm{BM}}(X,2)$,
    \[
        G \not \embeds (X', \vx)\,.
    \]
    Indeed, assume for contradiction there exists a pair $\{i,j\} \in
    \calL(\vx,X)$ with $\{i,j\}\in E(G)$, and a normed space $X' \in B_{\mathrm{BM}}(X,2)$
    such that $G \embeds (X',\vx)$. Then $\|x_i -x_j\|_{X'} \le \alpha$. As $X' \in B_{\mathrm{BM}}(X,2)$, it follows that $\|x_i - x_j\|_{X} \le 2\alpha $. By the second assertion of \cref{lemma:discretization} applied with $t:=2\alpha$, 
    \[
        \|\hat x_i -\hat x_j\|_{X} \leq 2\alpha+2 + 3c_0 \min(r_{\ell_i}, r_{\ell_j}) + 4c_0 \alpha< \frac{1}{3} \min(r_{\ell_i}, r_{\ell_j})\,,
    \]
    where in the last inequality we used 
    the bound $\min(r_{\ell_i}, r_{\ell_j})\geq 72\lambda= 72\alpha$ from Lemma~\ref{lemma:scale-lb}.
    This implies that $\{i,j\} \not \in \calL(\vx,X)$, which supplies the desired contradiction.

\medskip
    
    Utilizing the proven claim and then applying \eqref{eq:cardinality-long}, we compute:
    \begin{align*}
        &\PP[{\bf G}]{\exists\, (X',\vx) \in B_{\text{BM}}(X;\,2) \times \domain(n^2,\lambda,X) \,:\, {\bf G} \embeds (X',\vx)}\\
        &\quad\quad\le \PP[{\bf G}]{\exists\, \vx \in  \domain(n^2,\lambda,X) \,:\, E({\bf G})\cap \calL( \vx,X)=\emptyset} \\
        &\quad\quad\le\big(\pa{3/c_0}^d\big)^n\cdot \big(n^{1-\eps}\log^4 n\big)^n
        \cdot n^{n^{1-\eps/2}}\,
        \max\limits_{\vx \in  \domain(n^2,\lambda,X)}
        \PP[{\bf G}]{E({\bf G})\cap \calL ( \vx,X)=\emptyset}\,.
    \end{align*}
    
Fix any $\vx \in  \domain(n^2,\lambda,X)$.
As Lemma~\ref{lemma:scale-lb} implies $\frac{4}{9}\,r_{\ell_{i}}-2>\frac{1}{3}\,r_{\ell_i}$, observe that the third assertion of \cref{lemma:discretization} yields
\begin{equation}\label{calLuconstraints}
\big|\big\{j\in [n]\setminus\{i\}:\;\{i,j\}\notin\calL( \vx,X)\big\}\big|\leq L\,n^\eps=n^{2\eps},
\quad i\in[n].    
\end{equation}

\medskip
    
    In the remainder of this proof, let 
    $\mathbb P_{G(n,\Delta n / 2)}$ denote the uniform measure on
    the set $G(n,\Delta n / 2)$ of graphs on $[n]$ with $\Delta n /2$ edges. 
    We compute:
    \begin{align*}
        \PP[{\bf G}]{E({\bf G})\cap \calL ( \vx,X)=\emptyset} &= \PP[G(n,\Delta n / 2)]{E(G)\cap \calL( \vx,X)=\emptyset\,\big|\,G \in \dreg} \\
        &\le \frac{\PP[G(n,\Delta n / 2)]{
        E(G)\cap \calL( \vx,X)=\emptyset}}{\PP[G(n,\Delta n / 2)]{G \in \dreg}}\,.
    \end{align*}
    In the last line, the probability in the denominator is lower-bounded via \cref{cor:contiguity}. 
    We claim that
    \begin{equation}\label{eq:erdos-renyi}
        \PP[G(n,\Delta n / 2)]{ E(G)\cap \calL( \vx,X)=\emptyset} \le \cE{-\frac{(1-2\eps)}{2}\,n\,\Delta\log(n) + \aO{n\Delta}}\,.
    \end{equation}
    Indeed, note that $|E(G) \cap \calL(\vx,X)|$, $G\sim \mathbb P_{G(n,\Delta n / 2)}$ is distributed as a hypergeometric random variable with $n(n-1)/2$ total population, $|\calL(\vx,X)|$ population marked as ``success,'' and $n\Delta/2$ trials. Then
    \begin{align*}
        \PP[G(n,\Delta n / 2)]{ E(G)\cap \calL( \vx,X)=\emptyset} = \frac{
        \binom{\binom{n}{2} - |\calL(\vx,X)|}{n\Delta/2}}{\binom{\binom{n}{2}}{n\Delta/2}}.
    \end{align*}
    We have by \eqref{calLuconstraints} that $|\binom{n}{2}-|\calL(\vx,X)|| \le n^{1+2\eps}$. An application of the elementary inequality
    $$(m/k)^k \le \binom{m}{k} \le (me/k)^k, \quad \text{ for } 1\leq k\leq m,$$ yields
    \begin{align*}
        \PP[G(n,\Delta n / 2)]{ E(G)\cap \calL( \vx,X)=\emptyset} &\le \pa{\frac{e n^{1+2\eps}}{n(n-1)/2}}^{n\Delta/2} 
        = \cE{-\frac{1-2\eps}{2}\,n\,\Delta \log n + \aO{n\Delta}} \,.
    \end{align*}
    This establishes \eqref{eq:erdos-renyi}. Combining \eqref{eq:erdos-renyi} with the computations preceding it implies
    \begin{align*}
        \PP[{\bf G}]{E({\bf G})\cap \calL ( \vx,X)=\emptyset}  &
        \leq \cE{-\frac{1-2\eps}{2}\,n\,\Delta \log n + o(n\log n)} \,,
    \end{align*}
and hence, taking $c_{\text{\tiny{\ref{prop:non-embedding}}}} = \eps/10$ and recalling $\Delta \ge 3$,
\begin{align*}
&\PP[{\bf G}]{\exists\, (X',\vx) \in B_{\text{BM}}(X;\,2) \times \domain(n^2,\lambda,X) \,:\, {\bf G} \embeds (X',\vx)}\\
&\leq
\big(\pa{3/c_0}^d\big)^n\cdot \big(n^{1-\eps}\log^4 n\big)^n
\cdot n^{n^{1-\eps/2}}
\cE{-\frac{1-2\eps}{2}\,n\,\Delta \log n + o(n\log n)}
<\exp(-c_{\text{\tiny{\ref{prop:non-embedding}}}}\,n\log n).
\end{align*}
\end{proof}

\begin{remark}
Observe that the proof of Proposition~\ref{prop:non-embedding}
does not actually depend on the model of randomness of the graph
${\bf G}$, except for the number of edges it has and the condition
that the edges are ``well spread'' within ${[n]\choose 2}$.
In particular, replacing ${\bf G}$ with the Erdos--Renyi
model $G(n,\Delta n/2)$ would yield a similar result.
\end{remark}

\subsection{Proof of the main result}

Note that the following technical statement, combined with the upper--bound in \eqref{eq:state of the art}, readily implies \cref{thm:graphs} and thus \cref{thm:main}, completing the article. 

\begin{theorem}[Bi-Lipschitz non-embedding]\label{thm:bilipschitz}
Let $\Delta\geq 3$, let $n$ be a sufficiently
large integer with $\Delta n$ even,
and let dimension parameter $d\geq 1$ satisfy $300^d\leq n^\varepsilon/2$. 
Let $\alpha$ be defined by \eqref{alphadefinition}.
Then
\begin{align*}
    &\PP[{\bf G}]{\mbox{${\bf G}$
    embeds with bi-Lipschitz distortion at most
    $\alpha$ into some $d$--dimensional
    normed space}}\\
    &\leq \exp(-cn\log n),
\end{align*}
for some universal constant $c>0$.
\end{theorem}
\begin{proof}
Recall that $\mathcal B_d$ denotes the Banach--Mazur compactum, and let
$\netB(d)$ be the net in $\mathcal{B}_d$ defined in Lemma~\ref{lemma:BM-entropy} (in view of our requirements on $d$, we have
$|\netB(d)|\leq \exp(\sqrt{n})$).

To get the result, it is sufficient to prove that
$$
\PP[{\bf G}]{\exists\, (X',\vx)\in \mathcal{B}_d
\times (\RR^d)^n\;:\, {\bf G} \embeds (X',\vx)}
        =\ao{1}.
$$
Observe that, in view of Proposition~\ref{prop:non-sparseI},
every normed space $X'$ and every tuple such that
${\bf G} \embeds (X',\vx)$, must satisfy
$$
\big|\big\{j\neq i:\;\|x_j-x_i\|_{X'}\leq 2\lambda\big\}\big|
< n^{\varepsilon},\quad i\in[n],
$$
i.e., the tuple $\vx$ must be $2\lambda$--sparse with respect to $X'$.
If $X\in \netB(d)$ is any normed space with the Banach--Mazur
distance at most $2$ from $X'$ then the tuple $\vx$ is $\lambda$--sparse
in $X$. Furthermore, using translational invariance of the metric in a normed space, we can assume that the first component of the tuple $\vx$ is zero, so that in particular $\vx\subset B_X(0,\alpha n)\subset B_X(0,n^2)$.
Thus, we can write
\begin{align*}
&\PP[{\bf G}]{\exists\, (X',\vx)\in \mathcal{B}_d
\times (\RR^d)^n\;:\, {\bf G} \embeds (X',\vx)}\\
&\leq \sum\limits_{X\in \netB(d)}
\PP[{\bf G}]{\exists\, (X',\vx) \in B_{\text{BM}}(X;\,2) \times \domain(n^2,\lambda,X) \,:\, {\bf G} \embeds (X',\vx)}\\
&\leq \exp(-c_{\text{\tiny{\ref{prop:non-embedding}}}} \,n\log n+\sqrt{n}),
\end{align*}
where at the last step we applied the cardinality estimate for $\netB(d)$
and Proposition~\ref{prop:non-embedding}.
The result follows.
\end{proof}

\bigskip

\noindent {\bf An open problem.}
In this work, we do not consider the interesting regime of
{\it sublogarithmic} distortion.
Whereas essentially optimal bounds on $k_n^\alpha(\ell_\infty)$
are known for $\alpha=O\big(\frac{\log n}{\log \log n}\big)$
(see \cite{naor-snapshot} for details),
the following problem remains open as of this writing:
\begin{problem}
Find optimal bounds on the metric dimension reduction modulus
$k_n^\alpha(\ell_\infty)$ in the range
$$\frac{\log n}{\log \log n}\leq \alpha\leq \log n.$$
\end{problem}

\bibliography{references}
\bibliographystyle{alpha}

\appendix

\section{Discretization scheme: details}

\begin{proof}[Proof of Lemma~\ref{lemma:discretization}]
    We prove the items in the statement of the lemma in order. \\

    \textbf{Proof of claim 1}. By definition, $|\net_0|$ is the covering number of $B_X(0, D)$ using translates of $B_X(0,1)$. 
    By elementary volumetric considerations, along with the standard duality between covering and packing, 
    \[
        |\net_0| \le \pa{3 \, D }^d \le (3n)^{2n^{\eps/4}}\,.
    \]
    Recalling the definition of $\net(y,r_\ell;c_0r_\ell)$, we similarly have:
    \[
        \max_{\substack{y \in \RR^d \\ \ell \ge 0}}|\net(y,r_\ell;c_0r_\ell)| \le \pa{3/c_0}^d\,.
    \]
    Recall further that for every
    $\ell\in\{\lfloor \log_2\lambda\rfloor+1,\dots, \lceil\log_2(2D)\rceil\}$,
    $\hat S_\ell$ is a submultiset of $\net_0$ of size $\lfloor n^{1-\eps}\log^2 n\rfloor$.
    Fix for a moment
    any multisets $S_{\ell}^*\subset \net_0$ of size $\lfloor
    n^{1-\eps}\log^2 n\rfloor$, $\lfloor \log_2\lambda\rfloor+1\leq \ell\leq
    \lceil\log_2(2D)\rceil$,
    arbitrary numbers $\ell_i^*\in \{\lfloor \log_2\lambda\rfloor+1,\dots, \lceil\log_2(2D)\rceil\}$,
    and arbitrary elements $s_i^*\in S_{\ell_i}^*$, $i\leq n$, and consider the set
    $$
    \big\{\hx(\vx)\,:\, \vx \in \tuples(D,\lambda),\,\ell_i(\vx)=\ell_i^*,\,
    \hat s_i(\vx)=s_i^*,\;i\leq n;\;\;\hat S_\ell=S_\ell^*,\,\lfloor \log_2\lambda\rfloor+1\leq \ell\leq
    \lceil\log_2(2D)\rceil\big\}.
    $$
    By definition, for every $\vx \in \tuples(D,\lambda)$ with $\ell_i(\vx)=\ell_i^*$
    and $\hat s_i(\vx)=s_i^*$, the $i$--th component of $\hx$
    is the projection of $x_i$ onto the net $\net(s_i^*, r_{\ell_i^*}; c_0 r_{\ell_i^*})$ of size at most $\pa{3/c_0}^d$. Thus, the cardinality of the above set is at most
    $$
    \big(\pa{3/c_0}^d\big)^n.
    $$
    Further, given $\vx \in \tuples(D,\lambda)$ with $\ell_i(\vx)=\ell_i^*$, $i\leq n$,
    and $\hat S_\ell=S_\ell^*$, $\lfloor \log_2\lambda\rfloor+1\leq \ell\leq
    \lceil\log_2(2D)\rceil$,
    there are at most $\big(n^{1-\eps}\log^2 n\big)^n$
    admissible values for the tuple $\big(\hat s_i(\vx)\big)_{i=1}^n$, and hence
    $$
    \big|\big\{\hx(\vx)\,:\, \vx \in \tuples(D,\lambda),\,\ell_i(\vx)=\ell_i^*,\,
    \;\hat S_\ell=S_\ell^*,\,\lfloor \log_2\lambda\rfloor+1\leq \ell\leq
    \lceil\log_2(2D)\rceil\big\}\big|
    \leq \big(\pa{3/c_0}^d\big)^n\cdot \big(n^{1-\eps}\log^2 n\big)^n.
    $$
    The upper estimate on the cardinality of $\net_0$ allows for estimating
    the number of potential choices for sets $\hat S_\ell$, 
    leading to the bound
    \begin{align*}
    \big|\big\{\hx(\vx)\,:\, \vx \in \tuples(D,\lambda),\,\ell_i(\vx)=\ell_i^*,\;i\leq n
    \big\}\big|
    &\leq \big(\pa{3/c_0}^d\big)^n\cdot \big(n^{1-\eps}\log^2 n\big)^n
    \cdot |\net_0|^{(1+\lceil\log_2(2D)\rceil)
    \lfloor n^{1-\eps}\log^2 n\rfloor}\\
    &\leq \big(\pa{3/c_0}^d\big)^n\cdot \big(n^{1-\eps}\log^2 n\big)^n
    \cdot \big((3n)^{2n^{\eps/4}}\big)^{n^{1-\eps}\log^4 n}.
    \end{align*}
    Finally, in view of Remark~\ref{boundsonell},
    there are at most $\big(1+\lceil\log_2(2n^2)\rceil\big)^n$ possible realizations of $\ell_i(\vx)$, $i\leq n$.
    Combining the bounds, we get the claim. \\

    \textbf{Proof of claim 2.} Pick any $i,j$ with $i\neq j$. By triangle inequality and the construction of $\ell_j$:
    \[
        \ba{\cb{i' \in [n]\,:\, \|x_j - x_{i'}\|_X \le \|x_i-x_j\|_X + r_{\ell_i}}} \ge L\,n^{\eps}\,.
    \]
    As we have assumed $\|x_i - x_j\|_X \le t$, it holds by definition of $\ell_j$ that $r_{\ell_j} \le 2(r_{\ell_i} + t)$. 
    Then, by construction of the discretization $\hx$ (see Remark~\ref{distxitohatxi}),
    \begin{align*}
        \|\hat x_i - \hat x_j\|_X \le \|x_i - x_j\|_X + \|x_i - \hat x_i\|_X + \|x_j - \hat x_j\|_X  \le   t+2 + c_0 (r_{\ell_i} + r_{\ell_j})
        \leq t+2 + 3c_0 r_{\ell_i} + 2c_0 t. \\
    \end{align*}

    \textbf{Proof of claim 3.} 
    We prove the claim by contradiction. Assume that
    there is $i\in[n]$ such that
    \[
        \ba{\cb{j \in [n]\,:\, \|\hat x_i - \hat x_j\|_X \le \frac{4}{9}\,r_{\ell_{i}}-2}} > L\,n^{\eps}\,.
    \]
    Recall that, by the definition of $r_{\ell_i}$,
    \[
        \ba{\cb{j \in [n]\,:\, \|x_i - x_j\|_X \le r_{\ell_{i}}/2}} < L\,n^{\eps}\,.
    \]
    Therefore, there exists $x_{j_0}$ with both
    $\|x_i - x_{j_0}\| > r_{\ell_i}/2$ and $\|\hat x_i - \hat x_{j_0}\| \le \frac{4}{9}r_{\ell_i}-2$.
    By triangle inequality, one of the following must hold:
    \[
        \|\hat x_i - x_i \| \ge r_{\ell_i}/90+1\,, \quad \text{or} \quad \|\hat x_{j_0} - x_{j_0} \| \ge r_{\ell_i}/30+1\,.
    \]
    By construction of $\hx$ (see Remark~\ref{distxitohatxi}), we have that $\|\hat x_i - x_i\| \le c_0 r_{\ell_i}+1$ and $\|\hat x_{j_0} - x_{j_0}\| \le c_0 r_{\ell_{j_0}}+1$. Since $c_0 = .01$, the first of the two options above is impossible. Then the second option must hold, which implies
    \begin{equation}\label{akjnfaskfjn}
    c_0 r_{\ell_{j_0}}+1 \ge r_{\ell_i}/30+1.
    \end{equation}
    On the other hand, the definition of $\ell_{j_0}$
    implies that
    $$
    \log_2 r_{\ell_{j_0}}= \ell_{j_0}\leq\lceil\log_2(\|x_i - x_{j_0}\|_X+r_{\ell_i})\rceil
    \leq \Big\lceil\log_2\Big(\frac{4}{9}r_{\ell_i}+c_0 (r_{\ell_i}+r_{\ell_{j_0}})+r_{\ell_i}\Big)\Big\rceil,
    $$
    and hence $r_{\ell_{j_0}}< 3r_{\ell_i}$.
    Combining the last assertion with \eqref{akjnfaskfjn}, we get
    $$
    3c_0r_{\ell_i}\geq r_{\ell_i}/30,
    $$
    which contradicts our choice of $c_0$. This proves the claim.
\end{proof}

\bigskip

\begin{proof}[Proof of Lemma~\ref{lemma:scale-lb}]
    Assume for the sake of contradiction there is a $\lambda$-sparse tuple $\vx$ in $(B_{X}(0;\, D))^n$ and an index $i \in [n]$ such that
    \[
        \ba{\cb{j \in [n] \,:\, \|x_i - x_j\|_X \le 72\lambda}  } \ge L\,n^{\eps}\,,
    \]
    that is, the ball $B_{X}(x_i, 72\lambda)$ contains at least $L\,n^{\eps}$
    points from $\vx$. By the definition of $\lambda$--sparsity,
    any of the balls $B_X(x_j,\lambda)$, $j\in[n]$, contains at most
    $n^{\eps}+1$
    points from $\vx$. Hence, any covering of $B_{X}(x_i, 72\lambda)$ using translates of $B_X(0,\, \lambda/2)$ must have cardinality at least 
    $L\,n^{\eps}/(n^{\eps}+1)$, since by the pigeonhole principle, any smaller covering would have a $\lambda/2$--ball with more than $n^{\eps}+1$ points from $\vx$.
    By the standard duality between packing and covering, this implies that $B_{X}(x_i, 73\lambda)$ contains as a subset at least $L\,n^{\eps}/(n^{\eps}+1)$ disjoint copies of $B_X(0, \lambda/4)$. By volumetric considerations, we obtain:
    \[
        \frac{L\,n^{\eps}}{n^{\eps}+1}\, (\lambda/4)^{d} \le (73\lambda)^d\,.
    \]
    However, by our choices of parameters, $300^d \le \frac{L\,n^{\eps}}{n^{\eps}+1}$, providing the desired contradiction. Finally, note that by construction of $\ell_i$, we have $\ell_i \geq \log_2(72\lambda)$, so that $r_{\ell_i} \geq 72\lambda$.
\end{proof}

\end{document}